\documentclass[11pt]{amsart}
\usepackage[margin=1in]{geometry}
\usepackage{quiver}

\usepackage{amsmath,amsthm,amssymb}
\usepackage{euscript}
\usepackage{enumerate}
\usepackage{amsfonts}
\usepackage{bm}
\usepackage{hyperref}
\usepackage{xcolor}
\hypersetup{
    colorlinks,
    linkcolor={blue!80!pink},
    citecolor={blue!80!pink},
    urlcolor={blue!100!black}
}
\usepackage[capitalise]{cleveref}
\usepackage[colorinlistoftodos]{todonotes}


\newtheorem{theorem}{Theorem}[section]

\newtheorem{corollary}[theorem]{Corollary}

\newtheorem{remark}[theorem]{Remark}
\newtheorem{lemma}[theorem]{Lemma}

\renewcommand{\url}[1]{\href{#1}{Available online}}

\newcommand{\adjunction}[4]{%
  #1\colon #2%
  \mathrel{\vcenter{%
    \offinterlineskip\m@th
    \ialign{%
      \hfil$##$\hfil\cr
      \longrightharpoonup\cr
      \noalign{\kern-.3ex}
      \smallbot\cr
      \longleftharpoondown\cr
    }%
  }}%
  #3 \noloc #4%
}

\newcommand{\Sp}{\operatorname{Sp}}

\newcommand{\id}{\operatorname{id}}

\newcommand{\sphere}{\mathbb{S}}
\newcommand{\Map}{\operatorname{Map}}

\newcommand{\K}{\operatorname{K}}
\newcommand{\F}{\mathbb{F}}

\newtheorem{maintheorem}{Theorem}

\title{Tate Frobenius and chromatic blueshift for finite spectra}
\author{Pengkun Huang}
\date{\today}

\begin{document}

\begin{abstract}
    Let $p$ be a prime and $n\geq 1$. We prove that the Tate-Frobenius of the $\K(n)$-local sphere induces an equivalence between $L_{n-1}L_{\K(n)}Y$ and $(L_{\K(n)}Y)^{t\Sigma_p} \cong (L_nY)^{t\Sigma_p}$ for any finite $p$-local type $n-1$ spectrum $Y$ with trivial $\Sigma_p$-action.
\end{abstract}
\maketitle
\section{Introduction}
It has long been observed that the Tate construction exhibits chromatic blueshift: applying it to a periodic spectrum can lower its chromatic height. The earliest example is the theorem of Mahowald and Shick that, at the prime \(2\), $(L_n^fX)^{tC_2}$ is contractible for any finite spectrum $X$ of type $n$ \cite{MahowaldShick1988}. Later, Hovey and Sadofsky  \cite{HoveySadofsky1996} proved that the Tate construction of an $L_n$-local spectrum with trivial finite-group action is $L_{n-1}$-local. These results describe the chromatic behavior of the Tate construction, while leaving the homotopy type of the resulting spectrum to be determined.

A more precise prediction is provided by a conjecture of Hopkins and Mahowald, recorded in \cite{AndoMoravaSadofsky1998}. At the prime \(2\), it predicts that, for a finite spectrum \(Y\) of type \(n-1\),
\[
(L_nY)^{tC_2}
\simeq
L_{n-1}(Y_2^\wedge)
\vee
\Sigma^{-1}L_{n-1}(Y_2^\wedge).
\]
This prediction is closely related to the strong chromatic splitting conjecture, which concerns the lower-height localization \(L_{n-1}L_{\K(n)}Y\). At height two, the strong chromatic splitting conjecture holds at every odd prime: for $p\geq5$ it follows from the calculations of Shimomura--Yabe \cite{SY95}, revisited by Behrens \cite{Beh12}, while the case $p=3$ was established by Goerss--Henn--Mahowald
\cite[Theorem~5.11]{GHM14}.

However, the additional Moore-spectrum summands in the calculation at height two and the prime \(2\) by Beaudry, Goerss, and Henn \cite{BeaudryGoerssHenn2022} show that the corresponding two-summand description of the chromatic overlap should be modified at the prime \(2\). This suggests one could compare the Tate spectrum directly with the chromatic overlap, independently of a proposed decomposition given by chromatic splitting. The main result of this paper constructs such a comparison and proves that it is an equivalence

\begin{maintheorem}[\cref{main}] Let \(p\) be a prime and \(n\geq1\). For every finite \(p\)-local spectrum \(Y\) of type \(n-1\), the Tate-Frobenius of the \(K(n)\)-local sphere induces a natural equivalence
\[
\psi_Y\colon
L_{n-1}L_{\K(n)}Y
\xrightarrow{\;\simeq\;}
(L_{\K(n)}Y)^{t\Sigma_p}
\simeq
(L_nY)^{t\Sigma_p},
\]
where $\Sigma_p$ acts trivially on $Y$.
\end{maintheorem}
When $n=1$, this is proven independently in \cite{heutsmeiernikolaus}. Our proof begins with Hahn’s calculation \cite{hahn2016bousfield} of the power operation on Morava \(E\)-theory modulo
$I=(p,u_1,\ldots,u_{n-2})$.
This calculation shows that the Tate Frobenius gives the required equivalence after lower-height localization and reduction modulo \(I\). We extend the comparison to the higher terms of the completed Morava descent resolution, following a similar idea in \cite[section 2]{DevinatzHopkins2004}. The descendability of $E_n$ established in \cite{Mathew2016Galois} makes the associated totalization tower pro-constant, allowing the exact functors in the comparison to pass through descent. Finally, generalized Moore spectra and the thick subcategory theorem extend the comparison from the quotient \(E_n/I\) to every finite spectrum of type \(n-1\).
  
\section*{Acknowledgment}
The author is deeply grateful to Mark Behrens for all the helpful discussions. 

\section*{Statement of AI use}
This project originated from insight the author obtained through his computations \cite{huang2026equivariant} at $n=1$ and odd primes. ChatGPT-6 subsequently suggested a proof strategy in the case of $n=p=2$.  This strategy suggested the use of Mathew's observation \cite{Mathew2016Galois} that $S_{\K(n)} \to E_n$ is $\K(n)$-locally descendable, as well as Rezk's formula for power operations for $E_2$-theory at $p=2$ using its elliptic curve model \cite{rezk2008power}. The author developed and carried out the argument for all heights and primes. ChatGPT was also used to check details, locate references, and improve the exposition. The author takes full responsibility for the mathematical content of this paper.
  
\section*{Notation and Conventions}
\begin{itemize}
    \item Let $E_n$ be the Lubin-Tate theory associated to the Honda formal group over $k=\mathbb{F}_{p^n}$. Write $\pi_*E_n = W(k)[[u_1, \cdots, u_{n-1}]] [u^{\pm 1}]$, where $|u|=-2$.
    \item Let $V$ be the reduced permutation representation $\mathbb{C}^p/\mathbb{C}$ of $\Sigma_p$.
    \item For an even-periodic cohomology theory, we write $e_V \in E^0(B\Sigma_p)$ for the normalized Euler class of $V$.
    \item For an $E_n$-module $M$, the notation $E_n^{t\Sigma_p}\otimes_{E_n,\psi}M$ is used when the $E_n$-module structure of $E_n^{t\Sigma_p}$ is induced by the Tate-Frobenius map (\ref{tate-frobenius}).
\end{itemize}

\section{Tate-Frobenius map}
Let $R$ be an $\mathbb{E}_{\infty}$-ring. By \cite[Corollary IV.1.4]{nikolaus2018topological}, it admits the Tate-Frobenius map that factors through the $\F_p^{\times}$-homotopy fixed points as follows
     \begin{equation*}
        R \rightarrow (R ^{tC_p})^{h\F_p^{\times}}  \rightarrow  R^{tC_p},
     \end{equation*}
where $\mathbb{F}_p^{\times}$ acts through its automorphism action on $C_p$.

\begin{lemma} \label{CpandSigmap}
    Let $R$ be a $p$-local spectrum. Then we have natural equivalences
    \begin{equation*}
           (R^{hC_p})^{h\F_p^{\times}} \cong R^{h\Sigma_p} \quad \text{and} \quad 
        (R^{tC_p})^{h\F_p^{\times}} \cong R^{t\Sigma_p},
    \end{equation*}
    where $\Sigma_p$ and $C_p$ act trivially on $R$.
\end{lemma}
\begin{proof}
    Notice that 
    \begin{equation*}
        (R^{hC_p})^{h\F_p^{\times}} \cong F(BN_{\Sigma_p}(C_p)_+, R),
    \end{equation*} 
    where $N_{\Sigma_p}(C_p)$ is the normalizer of $C_p$ in $\Sigma_p$. The first claim follows from the fact that $ \Sigma^{\infty}_+ BN_{\Sigma_p}(C_p) \rightarrow  \Sigma^{\infty}_+ (B\Sigma_p)$ is a $p$-local equivalence \cite[Theorem B]{HarrisKuhn1988}. The second part is a result of \cite[Proposition A.18]{glasman2020stable}. We notice that proposition establishes the equivalence of Tate spectra for any $\Sigma_p$-spectrum.
\end{proof}

\begin{remark}\label{identifyassubring}
    Because $\mathbb{F}_p^{\times}$ is of order $p-1$ prime to $p$, we have an isomorphism for any $p$-local spectrum $R$
    \begin{equation*}
        \pi_*(R^{h\mathbb{F}_p^{\times}}) \cong (\pi_*R)^{\mathbb{F}_p^{\times}}
    \end{equation*}
    because the homotopy fixed point spectral sequence collapses. As a result, we can identify $\pi_*R^{h\Sigma_p}$ and $\pi_*R^{t\Sigma_p}$ as subgroups of $\pi_*R^{hC_p}$ and $\pi_*R^{tC_p}$, respectively.
\end{remark}

\begin{remark} \label{eulerclass}
    If we further assume that $R$ is complex oriented and let $i$ be the natural inclusion map $\mathbb{C}_p \rightarrow \Sigma_p$, then by the observation that the Euler class $e_{i^*V} = \prod_{k=1}^{p-1} [k](z)$ is an $\F_p^{\times}$-invariant, we can also learn that
    \begin{equation*}
        R^{t\Sigma_p} \cong  (R^{tC_p})^{h\F_p^{\times}} \cong (R^{hC_p}[e_{i^*V}^{-1}])^{h\F_p^{\times}} \cong (R^{hC_p})^{h\F_p^{\times}}[e_{V}^{-1}] \cong R^{h\Sigma_p}[e_V^{-1}].
    \end{equation*}
\end{remark}

By the above lemma, for a $p$-local $\mathbb{E}_{\infty}$-ring, we can recognize the Tate-Frobenius map as a map
\begin{equation}\label{tate-frobenius}
  \psi \colon  R \rightarrow R^{t\Sigma_p}.
\end{equation}
As further pointed out in \cite[Corollary IV.1.10]{nikolaus2018topological}, the map on $\pi_0$ factors as
\begin{equation*}
   \psi \colon \pi_0R \xrightarrow{P} R^0(B\Sigma_p) \xrightarrow{\operatorname{can}} \pi_0R^{t\Sigma_p},
\end{equation*}
where $P$ is the usual power operation for $R$.

\section{Tate-Frobenius map of Morava $E$-theory modulo $(p,\cdots, u_{n-2})$}

Our comparison uses the following consequence of Mathew’s descendability result \cite{Mathew2016Galois}:

\begin{lemma} \label{constantpro}
    Let $F \colon \Sp_{\K(n)} \rightarrow \Sp$ be a functor that preserves finite limits. Then we have
    \begin{equation*}
        F(L_{\K(n)} \sphere )\cong  \lim_{\Delta}(
      \begin{tikzcd}
	F(E_{n}) & {F(L_{\K(n)}(E_{n})^{\otimes 2})} & {F(L_{\K(n)}(E_{n})^{\otimes 3})} & \cdots
	\arrow[shift left=1, from=1-1, to=1-2]
	\arrow[shift right=1, from=1-1, to=1-2]
	\arrow[shift left=2, from=1-2, to=1-3]
	\arrow[shift right=2, from=1-2, to=1-3]
	\arrow[from=1-2, to=1-3]
	\arrow[shift left=3, from=1-3, to=1-4]
	\arrow[shift right=3, from=1-3, to=1-4]
	\arrow[shift left=1, from=1-3, to=1-4]
	\arrow[shift right=1, from=1-3, to=1-4]
\end{tikzcd}
).
    \end{equation*}
\end{lemma}

\begin{proof}
 By \cite[Proposition 10.10]{Mathew2016Galois}, the Morava $E$-theory $E_n$ is descendable in $\Sp_{\K(n)}$. \cite[Proposition 3.20]{Mathew2016Galois} states that it is equivalent to the $\operatorname{Tot}^n$ tower defined by the cosimplicial diagram
\[\begin{tikzcd}
	E_{n} & {L_{\K(n)}(E_{n})^{\otimes 2}} & {L_{\K(n)}(E_{n})^{\otimes 3}} & \cdots
	\arrow[shift left=1, from=1-1, to=1-2]
	\arrow[shift right=1, from=1-1, to=1-2]
	\arrow[shift left=2, from=1-2, to=1-3]
	\arrow[shift right=2, from=1-2, to=1-3]
	\arrow[from=1-2, to=1-3]
	\arrow[shift left=3, from=1-3, to=1-4]
	\arrow[shift right=3, from=1-3, to=1-4]
	\arrow[shift left=1, from=1-3, to=1-4]
	\arrow[shift right=1, from=1-3, to=1-4]
\end{tikzcd}\]
being a constant pro-object in $\Sp_{\K(n)}$ with limit being $L_{\K(n)}\sphere$. Then \cite[Proposition 3.10]{Mathew2016Galois} precisely gives the statement in the lemma.
\end{proof}
As suggested by the lemma above, we will prove that the Tate-Frobenius map induces a levelwise equivalence of the cosimplicial diagrams induced by the source and target in \cref{main}.

\begin{lemma} \label{flatnesslemma}
   Let $n\geq 2$ be an integer. The Tate-Frobenius map $ \psi \colon E_n \rightarrow E_n^{t\Sigma_p}$ induces a flat $\pi_*E_n$-module structure on $\pi_*E_n^{t\Sigma_p}$. Let $I=(p, \cdots, u_{n-2})$. Then we have an isomorphism
    $$
    \pi_* (E_n^{t\Sigma_p} \otimes_{E_n, \psi} E_n/I) \cong \pi_*(E_n^{t\Sigma_p})/I.
     $$
     Moreover, there is an equivalence of $E_n$-modules 
     $$
        E_n^{t\Sigma_p} \otimes_{E_n,\psi} E_n/I \cong E_n^{t\Sigma_p} \otimes_{E_n} E_n/I.
    $$
\end{lemma}

\begin{proof}
    Because $E_n^{t\Sigma_p}$ and $E_n$ are even periodic, we only need to show flatness for $\pi_0E_n^{t\Sigma_p}$. As we have seen in \cref{eulerclass}, 
    \begin{equation} \label{formula1}
        \pi_0R^{t\Sigma_p} \cong \pi_0R^{h\Sigma_p}[e_V^{-1}],
    \end{equation}
    where $e_V$ is the Euler class of the reduced $\Sigma_p$-representation on $\mathbb{C}^p/\mathbb{C}$ constructed by permutation of the coordinates. Consider the transfer ideal $I_{\operatorname{tr}} \subset R^0(B\Sigma_p)$ which is the image of the transfer map $R^0 \rightarrow R^0(B\Sigma_p)$. We have
    \begin{equation*}
        e_V \operatorname{Tr}_e^{\Sigma_p}(a)= \operatorname{Tr}_e^{\Sigma_p}(\operatorname{Res}_e^{\Sigma_p}(e_V) a) =0
    \end{equation*}
    because we have $\operatorname{Res}_e^{\Sigma_p}(e_V)=0$. Together with (\ref{formula1}), we get
    \begin{equation*}
        \pi_0R^{t\Sigma_p} \cong  \pi_0R^{h\Sigma_p}/I_{\operatorname{tr}}[e_V^{-1}].
    \end{equation*}
    Hence, we only need to show $R^0(B\Sigma_p)/I_{\operatorname{tr}}$ is flat along the power operation map. By \cite{strickland1997finite} (see also \cite[2.2.3]{Sukarto2026Power}), the power operation is realized geometrically as the target map
    \begin{equation*}
        t \colon \operatorname{Sub}_1{\mathbb{G}} \rightarrow \operatorname{Spf}E_0,
    \end{equation*}
    where $\operatorname{Sub}_1(G)$ is defined in section $10$ of \cite{strickland1997finite}, and Proposition $13.1$ in the same paper shows that this is a finite flat map. By flatness, the K\"unneth spectral sequence for calculating the tensor product must collapse, and we have
    \begin{equation*}
        \pi_* E_n^{t\Sigma_p} \otimes_{E_n, \psi} E_n/I \cong \pi_*E_n^{t\Sigma_p}/\psi(I).
    \end{equation*}
    By \cref{identifyassubring} and \cite[Lemma 2.5]{hahn2016bousfield}, we learn that $\psi(I)=I$ in $\pi_0E_n^{t\Sigma_p}$, which computes the homotopy groups as claimed. To show the module equivalence, all we need is to construct an $E_n$-module map that realizes an isomorphism of homotopy groups. Because 
        $E_n/I \cong E_n/p \otimes_{E_n} \cdots \otimes_{E_n} E_n/u_{n-2}$,
    we have
    \begin{equation*}
         E_n^{t\Sigma_p} \otimes_{E_n, \psi} E_n/I \cong  E_n^{t\Sigma_p}/\psi(p) \otimes_{E_n^{t\Sigma_p}} \cdots \otimes_{E_n^{t\Sigma_p}} E_n^{t\Sigma_p}/\psi(u_{n-2}).
    \end{equation*}
    The fact that $\psi(u_i) \in I$ induces an $E_n^{t\Sigma_p}$-module map $ E_n^{t\Sigma_p}/\psi(u_i) \rightarrow E_n^{t\Sigma_p}/I$. Using that $E_n^{t\Sigma_p}/I$ is an $E_n^{t\Sigma_p}$-algebra by \cite{HahnWilson2018Quotients}, these module maps assemble as follows
   \begin{equation*}
       E_n^{t\Sigma_p}/\psi(p)\otimes_{E_n^{t\Sigma_p}} \cdots \otimes_{E_n^{t\Sigma_p}} E_n^{t\Sigma_p}/\psi(u_{n-2}) \rightarrow E_n^{t\Sigma_p}/I \otimes_{E_n^{t\Sigma_p}} \cdots \otimes_{E_n^{t\Sigma_p}} E_n^{t\Sigma_p}/I \xrightarrow[]{\mu} E_n^{t\Sigma_p}/I.
   \end{equation*}    
\end{proof}

The main computational lemma comes from the weight $p$ power operation calculated by Hahn \cite{hahn2016bousfield}:

\begin{lemma} \label{level0eq}
    Let $n\geq 2$ and $I=(p,\cdots, u_{n-2})$ be the ideal in $\pi_0E_n$. Then we have
    \begin{equation*}
      \pi_* E_n^{t\Sigma_p}/I \cong k((e_V))[u^{\pm 1}].
    \end{equation*}
    The Tate-Frobenius map acts on the $\pi_0$-homotopy group modulo the ideal $I$ as
    \begin{equation*}
        \psi \colon \pi_0E_n/I \cong k[[u_{n-1}]] \rightarrow \pi_0E_n^{t\Sigma_p}/I \cong k((e_V)) \quad \psi(u_{n-1})=e_V  U,
    \end{equation*}
  where $U$ is a unit in $\pi_0E_n^{h\Sigma_p}/I$. Moreover, it induces an equivalence of $E_n$-modules
    \begin{equation*}
       \psi \colon  L_{n-1}E_n \otimes_{E_n} E_n/I \xrightarrow[]{\cong}(E_n)^{t\Sigma_p}\otimes_{E_n,\psi} E_n/I,
    \end{equation*}
where the module structure of $E_n^{t\Sigma_p}$ is induced by the Tate-Frobenius algebra map.
\end{lemma}

\begin{proof}
 By \cite[Proposition $4.3$]{hahn2016bousfield}, the power series $[p](z)$ of $E_n$ can be written as $[p](z)=z^{p^{n-1}}g_n(z)$ mod $I$ for some $g_n(z)$ and
 \begin{equation*}
     \pi_*E_n^{hC_p}/(I, g_n(z)) \cong k[[z]][u^{\pm1}]. 
 \end{equation*}
Hence, we have
\begin{equation*}
    \pi_*E_n^{tC_p}/I \cong \pi_*E_n[[z]]/([p](z),I)[z^{-1}] \cong \pi_*E_n^{hC_p}/(I, g_n(z))[z^{-1}] \cong k((z))[u^{\pm1}].
\end{equation*}
We notice that $a \in \mathbb{F}_p^{\times}$ acts on $z$ as $[a](z)$ and trivially on the periodic generator $u$, so the invariant subfield is $k((w))$, where $w=\prod_{k=1}^{p-1}[k](z)$, and $w$ is exactly the restricted $C_p$-Euler class of $V$. By \cref{identifyassubring} and \cref{eulerclass}, we obtain the first isomorphism.

By \cite[Lemma 2.5]{hahn2016bousfield}, we observe that the ideal $\psi(I)=I$ in $\pi_0E_n^{t\Sigma_p}$, so the Tate Frobenius map indeed induces a map as claimed, and its description is precisely Proposition $4.5$ of \cite{hahn2016bousfield} by renormalizing the valuation of $z$ to $1$. 

By \cref{flatnesslemma} and the fact that $(E_n)^{t\Sigma_p}$ is $L_{n-1}$-local \cite{HoveySadofsky1996}, we have the following induced map 
\begin{equation*}
   \psi \colon L_{n-1}E_n \otimes_{E_n} E_n/I \rightarrow (E_n)^{t\Sigma_p} \otimes_{E_n, \psi} E_n/I ,
\end{equation*}
where the $E_n$-module structure of $(E_n)^{t\Sigma_p}$ is induced by the Tate Frobenius map. For the source, we notice that by \cite[Proposition 3.7]{barthel2016centralizers}, 
\begin{equation*}
    L_{n-1}E_n/I \cong L_{K(n-1)}E_n/I \cong (E_n[u_{n-1}^{-1}])^{\wedge}_I/I,
\end{equation*}
which implies 
\begin{equation*}
    \pi_* L_{n-1}E_n/I \cong k((u_{n-1}))[u^{\pm 1}].
\end{equation*}
Notice that the Tate-Frobenius map is multiplicative, so $\psi(u)$ is a unit. Combining this with the description of $\psi(u_{n-1})$, we see the induced map between $L_{n-1}E_n/I$ and $(E_n/I)^{t\Sigma_p}$ must be an equivalence.
\end{proof}

\begin{corollary} \label{thickEnI1}
    Let $M$ be an $E_n$-module that is in the thick subcategory generated by $E_n/I$. Then the map
    \begin{equation*}
        \psi\otimes \operatorname{id} \colon L_{n-1}E_n \otimes_{E_n} M\rightarrow E_n^{t\Sigma_p} \otimes_{E_n, \psi} M
    \end{equation*}
    is an equivalence for $n \geq 2$.
\end{corollary}

\section{Extension to higher tensor powers}
\begin{lemma} \label{tateofhigher}
   Fix $n\geq 2$. For $q \geq 1$, set $C^q=L_{\K(n)} (E_n)^{\otimes q+1}$. Then we have
   $$
   \pi_*(C^q)^{t\Sigma_p}/I \cong\Map_{\operatorname{cts}}(\mathbb{G}_n^{q}, \F_{p^n}((e_V)))[u^{\pm 1}],
   $$ 
   where $\mathbb{G}_n$ is the extended Morava stabilizer group. Moreover, the induced Tate-Frobenius map
   \begin{equation*}
       L_{n-1}C^q \otimes_{E_n} E_n/I \rightarrow (C^q)^{t\Sigma_p} \otimes_{E_n, \psi} E_n/I
   \end{equation*}
   is an equivalence.
\end{lemma}

\begin{proof}
    By the same argument as in \cite[Lemma 4.2]{clausen2024descent}, using that $E_n^{h\Sigma_p}$ is a finite free $E_n$-module, we have
    \begin{equation} \label{moduleequi}
            (C^q)^{t\Sigma_p} \cong  C^q \otimes_{E_n} E_n^{t\Sigma_p}.
    \end{equation}
    By \cite[section 2]{DevinatzHopkins2004}, there is a natural isomorphism of $\pi_*E_n$-algebras
    \begin{equation*}
        \pi_* L_{\K(n)}E_n^{\otimes q+1} \cong \Map_{\operatorname{cts}}(\mathbb{G}_n^q, \pi_*E_n),
    \end{equation*}
    where the algebra structure map sends $x \in \pi_*E_n$ to $1\otimes \cdots \otimes 1 \otimes x$, which corresponds to the constant function with value \(x\).
     We notice that $I=(p,\cdots, u_{n-2})$ is a regular sequence in $\pi_*E_n$ and $\pi_*L_{\K(n)}E_n^{\otimes q+1} \cong \Map_{\operatorname{cts}}(\mathbb{G}_n^{q}, \pi_*E_n)$. Hence, we have
    \begin{equation*}
        \begin{aligned}
            \pi_* ((C^q)^{t\Sigma_p}/I)
            &\cong (\pi_*(C^q)^{t\Sigma_p})/I\\
            &\cong \pi_*(C^q \otimes_{E_n} E_n^{t\Sigma_p})/I \\
            &\cong (\Map_{\operatorname{cts}}(\mathbb{G}_n^q, \pi_*E_n) \otimes_{\pi_*E_n} \pi_*E_n^{t\Sigma_p})/I\\
            &\cong \Map_{\operatorname{cts}}(\mathbb{G}_n^q, \pi_*E_n/I) \otimes_{\pi_*E_n/I} \pi_*E_n^{t\Sigma_p}/I, \\    
        \end{aligned}
    \end{equation*}
    where the third isomorphism is due to the flatness of $\pi_*E_n^{t\Sigma_p}$ via the canonical map, and the last isomorphism holds because of the constant function algebra structure we described. By \cref{level0eq}, we learn the ring spectrum is even-periodic with $\pi_0$ equal to $\Map_{\operatorname{cts}}(\mathbb{G}_n^q, k[[u_{n-1}]]) \otimes_{k[[u_{n-1}]]} k((e_V))$.

    By recognizing $\pi_*E_n^{t\Sigma_p}$ as the $\mathbb{F}_p^{\times}$-invariant of $\pi_*E_n^{tC_p}$, we notice that
    \begin{equation*}
       \pi_* E_n^{t\Sigma_p}/I \cong (\pi_*E_n^{hC_p}/(I,g_n(z)))^{\mathbb{F}_p^{\times}}[e_V^{-1}],
    \end{equation*}
    where $g_n(z)$ is defined in the proof of \cref{level0eq}. And by Propositions $4.2$ and $4.3$ in \cite{hahn2016bousfield}, the canonical map $\pi_0E_n/I \rightarrow (\pi_0 E_n^{hC_p}/(I,g_n(z)))^{\mathbb{F}_p^{\times}} \cong k[[e_V]]$ maps $u_{n-1}$ to $(e_V)^{p^{n-1}} U'$, where $U'$ is a unit. Hence, $k[[e_V]]$ is a finite free $k[[u_{n-1}]]$-module, and we have
    \begin{equation*}
            \pi_0( (C^q)^{t\Sigma_p}/I) 
            \cong \Map_{\operatorname{cts}}(\mathbb{G}_n^q, k[[e_V]])[e_V^{-1}],
    \end{equation*}
    which is a natural subring of $\Map_{\operatorname{cts}}(\mathbb{G}_n^q, k((e_V)))$. Because $\mathbb{G}_n^q$ is compact, any continuous function $\mathbb{G}_n^q \rightarrow k((e_V))$ factors through $e_V^{-N}k[[e_V]]$ for some sufficiently large $N$. We then obtain the isomorphism as claimed.

    For the equivalence statement, we again use \cite[Proposition 3.7]{barthel2016centralizers} to compute $\pi_0(L_{n-1}C^q/I)$ as follows
  \begin{equation*}
      \pi_0L_{n-1}C^q/I \cong \pi_0L_{K(n-1)}C^q/I\cong \Map_{\operatorname{cts}}(\mathbb{G}_n^q, k[[u_{n-1}]])[u_{n-1}^{-1}] \cong \Map_{\operatorname{cts}}(\mathbb{G}_n^q, k((u_{n-1}))),
  \end{equation*}
  where the last isomorphism holds by the compactness of $\mathbb{G}_n^q$ as we argued above. Because $\psi$ is multiplicative, we know $\psi(u)$ is a unit, so we only need to consider the Tate-Frobenius map on $\pi_0$:
  \begin{equation*}
      \psi_0 \colon \Map_{\operatorname{cts}}(\mathbb{G}_n^q, k((u_{n-1}))) \rightarrow \Map_{\operatorname{cts}}(\mathbb{G}_n^q, k((e_V))).
  \end{equation*}
To understand this map, we recall that for any $g \in \mathbb{G}_n^q$, the algebra morphism 
  \begin{equation*}
      \operatorname{ev}_g \colon \pi_*C^{q} \cong \Map_{\operatorname{cts}}(\mathbb{G}_n^q, \pi_*E_n) \rightarrow \pi_*E_n, \quad  f \mapsto f(g)
  \end{equation*}
  can be realized as an $\mathbb{E}_{\infty}$-morphism $\operatorname{ev}_g \colon C^{q} \rightarrow E_{n}$ by \cite[Theorem 4.1, Lemma 4.7]{DevinatzHopkins2004}. Because the Tate-Frobenius map is constructed as a natural morphism on $\mathbb{E}_{\infty}-$rings, we have the following commutative diagram:
\[\begin{tikzcd}
	{C^{q}} & {(C^{q})^{t\Sigma_p}} & {C^{q} \otimes_{E_n} (E_n)^{t\Sigma_p}} \\
	{E_n} & {E_n^{t\Sigma_p}} & {E_n\otimes_{E_n} (E_n)^{t\Sigma_p}}
	\arrow["\psi", from=1-1, to=1-2]
	\arrow["{\operatorname{ev}_g}", from=1-1, to=2-1]
	\arrow["{(\operatorname{ev}_g)^{t\Sigma_p}}", from=1-2, to=2-2]
	\arrow["\cong"', from=1-3, to=1-2]
	\arrow["{\operatorname{ev}_g \otimes \id}", from=1-3, to=2-3]
	\arrow["\psi", from=2-1, to=2-2]
	\arrow[equals, from=2-2, to=2-3]
\end{tikzcd}\]
Together with (\ref{moduleequi}) and \cref{level0eq}, we obtain a commutative diagram of $E_n$-modules
\[\begin{tikzcd}
	{L_{n-1}C^{q}\otimes_{E_n} E_n/I} & {(C^{q})^{t\Sigma_p}\otimes_{E_n,\psi}E_n/I} & {C^{q} \otimes_{E_n} (E_n)^{t\Sigma_p} \otimes_{E_n} E_n/I} \\
	{L_{n-1}E_n \otimes_{E_n} E_n/I} & {E_n^{t\Sigma_p} \otimes_{E_n,\psi}E_n/I} & {E_n\otimes_{E_n} (E_n)^{t\Sigma_p} \otimes_{E_n}E_n/I}
	\arrow["\psi", from=1-1, to=1-2]
	\arrow["{\operatorname{ev}_g\otimes \operatorname{id}}", from=1-1, to=2-1]
	\arrow["\cong", from=1-2, to=1-3]
	\arrow["{(\operatorname{ev}_g)^{t\Sigma_p} \otimes \operatorname{id}}", from=1-2, to=2-2]
	\arrow["{\operatorname{ev}_g \otimes \id \otimes \operatorname{id}}", from=1-3, to=2-3]
	\arrow["\psi", from=2-1, to=2-2]
	\arrow["\cong", from=2-2, to=2-3]
\end{tikzcd}\]
Applying $\pi_0$ to the diagram, we find that for any $f \in \Map_{\operatorname{cts}}(\mathbb{G}_n^q, k((u_{n-1})))$ and any $g \in \mathbb{G}_n^q$, we have 
\begin{equation*}
      \psi(f)(g)=\psi(f(g)). 
  \end{equation*}
This implies that the Tate-Frobenius map $\psi$ on higher levels is simply postcomposition with $\psi$ at level $0$. Hence, the equivalence at level $0$ implies that the higher-level map is an equivalence for $C^q$ with $q \geq 1$.
\end{proof}

\begin{corollary} \label{thickEnI2}
     Let $M$ be an $E_n$-module that is in the thick subcategory generated by $E_n/I$. Then the map
    \begin{equation*}
        \psi\otimes \operatorname{id} \colon L_{n-1}C^q \otimes_{E_n} M\rightarrow (C^q)^{t\Sigma_p} \otimes_{E_n, \psi} M
    \end{equation*}
    is an equivalence for $n\geq 2$.
\end{corollary}

\section{Proof of the main theorem}
\begin{theorem}  \label{main}
   Let \(p\) be a prime and \(n\geq 1\). For every finite \(p\)-local spectrum \(Y\) of type \(n-1\), the Tate-Frobenius of the \(K(n)\)-local sphere induces a natural equivalence
\[
\psi_Y\colon
L_{n-1}L_{K(n)}Y
\xrightarrow{\;\simeq\;}
(L_{K(n)}Y)^{t\Sigma_p}
\simeq
(L_nY)^{t\Sigma_p},
\]
where $\Sigma_p$ acts trivially on $Y$.
\end{theorem}

\begin{proof}
First, notice that the second equivalence comes from the chromatic fracture square and we have $(L_{n-1}Y)^{t\Sigma_p}=(L_{n-1}L_{K(n)}\sphere \otimes Y)^{t\Sigma_p}=0$ by \cite[Corollary 3.2]{HoveySadofsky1996}.
 
Because $L_{\K(n)}\sphere$ is a $p$-complete $\mathbb{E}_{\infty}$-ring, it admits the Tate-Frobenius map as we discussed. Tensoring the morphism with identity maps for any finite type $n-1$ spectrum $Y$, we obtain the natural morphisms in the theorem.
   
We observe that it suffices to prove it for one type $n-1$-complex. Indeed, the collection of finite spectra such that $\psi_Y$ is an isomorphism forms a thick subcategory, and the reduction follows from the thick subcategory theorem. If $Y$ has at least type $n$, we notice that both sides must be $0$. Hence, let us consider a type $n-1$ generalized Moore spectrum
   \begin{equation*}
       X= \sphere/(p^{i_0}, v_1^{i_1}, \cdots, v_{n-2}^{i_{n-2}}).
   \end{equation*}
When $n=1$, we can choose $X=\sphere$ and the map is simply $\psi \colon L_{\K(1)}\sphere[1/p] \rightarrow (L_{\K(1)}\sphere)^{t\Sigma_p}$, where both sides are equivalent to $H\mathbb{Q}_p \oplus\Sigma^{-1}H\mathbb{Q}_p$, and the fact that $\psi$ is an equivalence follows by descent from the calculation of $\psi$ on $\operatorname{KU}_p^{\wedge}$ as explained in \cite[Proposition IV.1.12] {nikolaus2018topological}.

Because $L_{n-1}(-) \otimes X$ and $(-)^{t\Sigma_p}\otimes X$ are both exact functors, we learn from \cref{constantpro} that it is enough to show the map between cosimplicial diagrams induced by $\psi_X$ is a levelwise equivalence. Write
   $
       C^q \colon= L_{\K(n)}(E_n)^{\otimes (q+1)} 
   $
   and $I=(p,\cdots, u_{n-2})$. Then at level $q$, the map $\psi_X$ induces the Tate-Frobenius map on $C^q$ as
   \begin{equation*}
      \psi\otimes \id \colon L_{n-1}C^q \otimes_{E_n} (E_n\otimes X) \rightarrow (C^q)^{t\Sigma_p}\otimes_{E_n, \psi} (E_n\otimes X).
  \end{equation*}
  By \cref{thickEnI1} and \ref{thickEnI2}, we only need to show $E_n\otimes X$ is in the thick subcategory generated by $E_n/I$, and the theorem is proven because we get a levelwise equivalence of the cosimplicial diagrams. By construction of X, we know
      $BP\otimes X \cong BP/(p^{i_0}, v_1^{i_1}, \cdots, v_{n-2}^{i_{n-2}}).$
  After base change along the map $BP \rightarrow E_n$ which sends $v_i$ to $u_iu^{-p^i+1}$, we see
  \begin{equation*}
      E_n \otimes X \cong E_n/(p^{i_0}, u_1^{i_1}, \cdots, u_{n-2}^{i_{n-2}}).
  \end{equation*}
  To construct this module from $E_n/I$, let us write 
  $$E_n[u_k^l]=E_n/(p^{i_0}, \cdots, u_{k-1}^{i_{k-1}}, u_{k}^{l}, u_{k+1}, \cdots,u_{n-2})$$ for some $0 \leq k \leq n-2 $ and $l\geq 2$. Then there is a cofiber sequence of $E_n$-modules
  \begin{equation*}
      E_n[u_k]\rightarrow E_n[u_k^{l}] \rightarrow  E_n[u_k^{l-1}]
  \end{equation*}
  induced by the diagram (write $E_n[\bar{u}_{k}] = E_n/(p^{i_0},\cdots, u_{k-1}^{i_{k-1}}, u_{k+1}, \cdots, u_{n-2})$)
\[\begin{tikzcd}
	{E_n[\bar{u}_{k}]} & {E_n[\bar{u}_{k}]} & {E_n[u_k] } \\
	{E_n[\bar{u}_{k}]} & {E_n[\bar{u}_{k}]} & {E_n[u_k^l] } \\
	0 & {E_n[u_k^{l-1}] } & {E_n[u_k^{l-1}] }
	\arrow["{u_{k}}", from=1-1, to=1-2]
	\arrow["{\operatorname{id}}"', from=1-1, to=2-1]
	\arrow[from=1-2, to=1-3]
	\arrow["{u_{k}^{l-1}}"', from=1-2, to=2-2]
	\arrow[from=1-3, to=2-3]
	\arrow["{u_{k}^{l}}", from=2-1, to=2-2]
	\arrow[from=2-1, to=3-1]
	\arrow[from=2-2, to=2-3]
	\arrow[from=2-2, to=3-2]
	\arrow[from=2-3, to=3-3]
	\arrow[from=3-1, to=3-2]
	\arrow[from=3-2, to=3-3]
\end{tikzcd}\]
so the module $E_n\otimes X$ can be built inductively starting from $E_n/I$, which shows it is in the thick subcategory and completes the proof.
\end{proof}

\bibliographystyle{amsalpha}
\bibliography{ref}

\end{document}